\documentclass{amsart}
\usepackage{amsmath,amssymb,amsthm}
\usepackage[english]{babel}
\usepackage{cite}
\usepackage[utf8]{inputenc}
\newtheorem{theorem}{Theorem}
\newtheorem{lemma}[theorem]{Lemma}

\title{Effective Upper Bound Estimates for $|\zeta'(1/2+it)|$ via Exponential Sums}
\author{Ting, Liu}
\address{Department of Mathematics, China Three Gorges University, Yichang, 443002, China}
\email{2460212166@qq.com}

\author{ Jinjin, Ma}
\address{Department of Mathematics, China Three Gorges University, Yichang, 443002, China}
\email{2182476928@qq.com}

\author{BinJie, Chang}
\address{Department of Mathematics, China Three Gorges University, Yichang, 443002, China}
\email{1544259484@qq.com}

\author{Xinhua Xiong}
\address{Department of Computer Sciences and Information, China Three Gorges University, Yichang, 443002, China}
\email{xinhuaxiong@ctgu.edu.cn}

\begin{document}
\begin{abstract}
In this paper, we use methods of exponential sums to derive a formula for estimating effective upper bounds of $|\zeta'(1/2+it)|$. Different effective upper bounds can be obtained by choosing different parameters.
\end{abstract}
\maketitle


\section{Introduction}
The Riemann zeta function is defined by
\[
\zeta(s) = \sum_{n=1}^{\infty} n^{-s} \quad (\sigma > 1), \quad s = \sigma + it.
\]
Using the Euler–Maclaurin summation formula, we have
\begin{equation}\label{30}
\begin{aligned}
\zeta(s) &= \sum_{n=1}^{N-1} n^{-s} + \frac{N^{1-s}}{s-1} + \frac{1}{2}N^{-s} + \frac{B_2}{2}sN^{-s-1} + \cdots \\
&\quad + \frac{B_{2v}}{(2v)!}s(s+1)\cdots(s+2v-2)N^{-s-2v+1} + R_{2v},
\end{aligned}
\end{equation}
where $B_m$ are Bernoulli polynomials, $N$ is a positive integer, $v \geq 0$ is a natural number, and
\[
R_{2v} = -\frac{s(s+1)\cdots(s+2v-1)}{(2v)!} \int_N^{\infty} \bar{B}_{2v}(x)x^{-s-2v}dx.
\]
The integral for $R_{2v}$ converges, and \eqref{30} holds in the half-plane $\text{Re}(s+2v+1) > 1$.

When $v=0$, the formula simplifies to
\[
\zeta(s) = \frac{s}{s-1} - s \int_1^{\infty} \frac{u - [u]}{u^{s+1}} du,
\]
and
\[
\zeta(s) = \frac{1}{s} + \frac{1}{2} - s \int_1^{\infty} \frac{u - [u] - \frac{1}{2}}{u^{s+1}} du.
\]

In 1985, Titchmarsh \cite{5} proved
\[
\zeta'(\sigma + it) \le \exp\left( \frac{C \log t}{\log \log t} \right), \quad \sigma \geq 1/2, \quad t \geq 10.
\]
For effective upper bounds of the zeta function and its derivatives, see the references.

In this paper, using exponential sum methods, we obtain effective estimates for $|\zeta'(1/2+it)|$:

\begin{theorem}\label{theorem1.1}
If $t \geq e^2$, then
\[
|\zeta'(1/2+it)| \leq 2t^{1/2} \log t - 4t^{1/2} + 8.047 \log t + 6.399.
\]
\end{theorem}

\begin{theorem}\label{theorem1.2}
If $t \geq e^6$, then
\[
|\zeta'(1/2+it)| \leq Q_1 t^{1/6} (\log t)^2 + Q_2 t^{1/6} \log t + Q_3 t^{1/6} + Q_4 (\log t)^2 + Q_5 \log t + Q_6,
\]
where $Q_1(\tau,q,t_2), Q_2(k,\tau,q,t_2), Q_3(k,\tau,q,t_2), Q_4(\tau,t_2), Q_5(k,\tau,q,t_2), Q_6(k,\tau,q,t_1,t_2)$ are constants, and $k, \tau, q, t_1, t_2$ are parameters to be determined.
\end{theorem}

To estimate $|\zeta'(1/2+it)|$, we use the representation
\[
\zeta'(s) = -\sum_{n=1}^N \frac{\log n}{n^s} - \int_N^{\infty} \frac{u - [u]}{u^{s+1}} du + s \int_N^{\infty} \frac{u - [u]}{u^{s+1}} \log u \, du - \frac{1}{(s-1)^2 N^{s-1}} - \frac{\log N}{(s-1)N^{s-1}},
\]
and divide the exponential sum part into $\sum_{n \leq t^{1/3}} \frac{\log n}{n^s}$, $\sum_{t^{1/3} < n \leq t^{2/3}} \frac{\log n}{n^s}$, $\sum_{t^{2/3} < n \leq t} \frac{\log n}{n^s}$, $\sum_{t < n \leq N} \frac{\log n}{n^s}$. In Section 2, we prove some lemmas and estimate $\sum_{t < n \leq N} \frac{\log n}{n^s}$. In Section 3, we directly integrate to obtain an effective upper bound for $|\zeta'(1/2+it)|$. In Section 4, using exponential sum methods, we estimate $\sum_{t^{1/3} < n \leq t^{2/3}} \frac{\log n}{n^s}$ and $\sum_{t^{2/3} < n \leq t} \frac{\log n}{n^s}$, and combine the results to get the final effective upper bound for $|\zeta'(1/2+it)|$.

\section{Lemmas}
\begin{lemma}\label{lemma2.1} \cite[Lemma 1]{6}
If $f$ and $g$ are differentiable real-valued functions, and $f > 0$ is decreasing (or $f < 0$ is increasing), then
\begin{equation}\label{31}
\left| \int_a^b f(x) g'(x) dx \right| = 2 |f(a)| \max_{a \leq x \leq b} |g(x)|.
\end{equation}
\end{lemma}

\begin{lemma}\label{lemma2.2}
Let $b > a > \frac{t}{2\pi}$. For any positive integer $V$ and $\sigma \geq 0$, we have
\begin{subequations}
\begin{equation}\label{2.2a}
\left| \int_a^b \frac{\sin(t \log x \pm 2\pi v x)}{x^{1+\sigma}} \log x \, dx \right| \leq \frac{2 \log a}{a^\sigma (2\pi v a \pm t)},
\end{equation}
where the ``$\pm$'' signs are all ``$+$'' or all ``$-$''. Moreover,
\begin{equation}\label{2.2b}
\left| \int_a^b \frac{\sin(t \log x + 2\pi v x) \pm \sin(t \log x - 2\pi v x)}{x^{1+\sigma}} \log x \, dx \right| \leq \frac{8\pi v a \log a}{a^\sigma (4\pi^2 v^2 a^2 - t^2)}.
\end{equation}
If we replace the sine functions with cosine functions, from inequality \eqref{2.2b} we obtain
\begin{equation}\label{2.2c}
\left| \int_a^b \frac{\sin(t \log x) \sin(2\pi v x)}{x^{1+\sigma}} \log x \, dx \right| \leq \frac{8\pi v a \log a}{a^\sigma (4\pi^2 v^2 a^2 - t^2)},
\end{equation}
and
\begin{equation}\label{2.2d}
\left| \int_a^b \frac{\cos(t \log x) \sin(2\pi v x)}{x^{1+\sigma}} \log x \, dx \right| \leq \frac{8\pi v a \log a}{a^\sigma (4\pi^2 v^2 a^2 - t^2)}.
\end{equation}
\end{subequations}
\end{lemma}

\begin{proof}
The integral on the left-hand side of \eqref{2.2a} is
\[
- \int_a^b \frac{\log x}{x^{1+\sigma}} \frac{d\{ \cos(t \log x \pm 2\pi v x) \}}{(t/x \pm 2\pi v)} = - \int_a^b \frac{\log x}{x^\sigma (t \pm 2\pi v x)} d\{ \cos(t \log x \pm 2\pi v x) \}.
\]
For $x \geq a > \frac{t}{2\pi}$, clearly $f(x) = \frac{\log x}{x^\sigma (2\pi v x \mp t)} > 0$ and decreasing. Replacing $f(x)$ with $-f(x)$ in Lemma \ref{lemma2.1}, we get
\[
\left| \int_a^b \frac{\sin(t \log x \pm 2\pi v x)}{x^{1+\sigma}} \log x \, dx \right| \leq \frac{2 \log a}{a^\sigma (2\pi v a \pm t)}.
\]
Next, to prove \eqref{2.2b},
\begin{align*}
&\left| \int_a^b \frac{\sin(t \log x + 2\pi v x) \pm \sin(t \log x - 2\pi v x)}{x^{1+\sigma}} \log x \, dx \right| \\
&\quad \leq \left| \int_a^b \frac{\sin(t \log x + 2\pi v x)}{x^{1+\sigma}} \log x \, dx \right| + \left| \int_a^b \frac{\sin(t \log x - 2\pi v x)}{x^{1+\sigma}} \log x \, dx \right| \\
&\quad \leq \frac{8\pi v a \log a}{a^\sigma (4\pi^2 v^2 a^2 - t^2)},
\end{align*}
where
\[
\frac{1}{2\pi v a - t} + \frac{1}{2\pi v a + t} = \frac{4\pi v a}{4\pi^2 v^2 a^2 - t^2}.
\]
For \eqref{2.2c} and \eqref{2.2d}, using product-to-sum identities
\[
\cos(t \log x) \sin(2\pi v x) = \frac{1}{2} \left( \sin(t \log x + 2\pi v x) - \sin(t \log x - 2\pi v x) \right),
\]
\[
\sin(t \log x) \sin(2\pi v x) = \frac{1}{2} \left( -\cos(t \log x + 2\pi v x) + \cos(t \log x - 2\pi v x) \right),
\]
and combining with \eqref{2.2b}, we obtain
\[
\left| \int_a^b \frac{\sin(t \log x) \sin(2\pi v x)}{x^{1+\sigma}} \log x \, dx \right| \leq \frac{8\pi v a \log a}{a^\sigma (4\pi^2 v^2 a^2 - t^2)},
\]
and
\[
\left| \int_a^b \frac{\cos(t \log x) \sin(2\pi v x)}{x^{1+\sigma}} \log x \, dx \right| \leq \frac{8\pi v a \log a}{a^\sigma (4\pi^2 v^2 a^2 - t^2)}.
\]
This completes the proof of Lemma \ref{lemma2.2}.
\end{proof}

Taking $v=0$ in \eqref{30},
\[
\zeta(s) = \sum_{n=1}^N \frac{1}{n^s} - s \int_N^{\infty} \frac{u - [u]}{u^{s+1}} du + \frac{1}{(s-1)N^{s-1}}.
\]
Differentiating,
\begin{equation}\label{3}
\begin{aligned}
\zeta'(s) &= -\sum_{n=1}^N \frac{\log n}{n^s} - \int_N^{\infty} \frac{u - [u]}{u^{s+1}} du + s \int_N^{\infty} \frac{u - [u]}{u^{s+1}} \log u \, du \\
&\quad - \frac{1}{(s-1)^2 N^{s-1}} - \frac{\log N}{(s-1)N^{s-1}} \\
&= -\sum_{n=1}^N \frac{\log n}{n^s} + A + B + C + D,
\end{aligned}
\end{equation}
where
\begin{align*}
A &:= -\int_N^{\infty} \frac{u - [u]}{u^{s+1}} du, \\
B &:= s \int_N^{\infty} \frac{u - [u]}{u^{s+1}} \log u \, du, \\
C &:= -\frac{1}{(s-1)^2 N^{s-1}}, \\
D &:= -\frac{\log N}{(s-1)N^{s-1}}.
\end{align*}
Let $E(s) := A + B + C + D$.

\begin{lemma}\label{lemma2.3}
If $\sigma = 1/2$, $N = [t^2]$, then
\[
|E(s)| \leq |A| + |B| + |C| + |D| \leq \frac{2}{\sqrt{t^2 - 1}} + \sqrt{\frac{4t^2 + 1}{t^2 - 1}} (2 \log t + 2) + \frac{1}{t} + 2 \log t.
\]
\end{lemma}

\begin{proof}
For $\sigma > 0$, $s = \sigma + it$,
\[
|A| = \left| \int_N^{\infty} \frac{u - [u]}{u^{s+1}} du \right| \leq \int_N^{\infty} \frac{1}{u^{\sigma+1}} du = \frac{1}{\sigma N^\sigma},
\]
\[
|B| = \left| s \int_N^{\infty} \frac{u - [u]}{u^{s+1}} \log u \, du \right| \leq |s| \int_N^{\infty} \frac{\log u}{u^{\sigma+1}} du \leq \sqrt{(t/\sigma)^2 + 1} \frac{(\sigma \log N + 1)}{\sigma N^\sigma},
\]
\[
|C| = \left| \frac{1}{(s-1)^2 N^{s-1}} \right| \leq \frac{N^{1-\sigma}}{t^2},
\]
\[
|D| = \left| \frac{\log N}{(s-1)N^{s-1}} \right| \leq \frac{N^{1-\sigma} \log N}{t}.
\]
Since $\sigma = 1/2$, $t^2 - 1 < N = [t^2] \leq t^2$, we have
\begin{equation}\label{4}
\begin{aligned}
|E(s)| &\leq \frac{1}{\sigma N^\sigma} + \sqrt{(t/\sigma)^2 + 1} \frac{(\sigma \log N + 1)}{\sigma N^\sigma} + \frac{N^{1-\sigma}}{t^2} + \frac{N^{1-\sigma} \log N}{t} \\
&\leq \frac{2}{\sqrt{t^2 - 1}} + \sqrt{\frac{4t^2 + 1}{t^2 - 1}} (2 \log t + 2) + \frac{1}{t} + 2 \log t.
\end{aligned}
\end{equation}
This completes the proof of Lemma \ref{lemma2.3}.
\end{proof}

From \eqref{3},
\begin{equation}\label{5}
\begin{aligned}
\zeta'(s) &= -\sum_{n=1}^N \frac{\log n}{n^s} + E(s) \\
&= -\sum_{n=1}^{[t]} \frac{\log n}{n^s} - \sum_{t < n \leq t^2} \frac{\log n}{n^s} + E(s).
\end{aligned}
\end{equation}
We now estimate the second sum on the right-hand side using integration.

\begin{lemma}\label{lemma2.4}
If $\sigma = 1/2$, $t \geq e^2$, then
\[
\left| \sum_{t < n \leq t^2} \frac{\log n}{n^s} \right| \leq 2 \log t + 1.944.
\]
\end{lemma}

\begin{proof}
By the Euler summation formula,
\[
\sum_{a < n \leq b} \Phi(n) = \int_a^b \Phi(x) dx + \int_a^b \left(x - [x] - \frac{1}{2}\right) \Phi'(x) dx + \left(a - [a] - \frac{1}{2}\right) \Phi(a) - \left(b - [b] - \frac{1}{2}\right) \Phi(b),
\]
where $((x)) = x - [x] - \frac{1}{2}$ and $|((x))| \leq \frac{1}{2}$. Then
\begin{equation}\label{6}
\begin{aligned}
\sum_{t < n \leq t^2} \frac{\log n}{n^s} &= \int_t^{t^2} \frac{\log x}{x^s} dx + \frac{((t)) \log t}{t^s} - \frac{((t^2)) \log t^2}{t^{2s}} \\
&\quad + \int_t^{t^2} \frac{((x))}{x^{s+1}} dx - s \int_t^{t^2} \frac{((x)) \log x}{x^{s+1}} dx \\
&= I_1 + I_2 + I_3 + I_4 + I_5,
\end{aligned}
\end{equation}
where
\begin{align*}
I_1 &:= \int_t^{t^2} \frac{\log x}{x^s} dx, \\
I_2 &:= \frac{((t)) \log t}{t^s}, \\
I_3 &:= -\frac{((t^2)) \log t^2}{t^{2s}}, \\
I_4 &:= \int_t^{t^2} \frac{((x))}{x^{s+1}} dx, \\
I_5 &:= -s \int_t^{t^2} \frac{((x)) \log x}{x^{s+1}} dx.
\end{align*}
For $\sigma = 1/2$, $t \geq e^2$,
\begin{align*}
|I_1| &= \left| \int_t^{t^2} \frac{\log x}{x^s} dx \right| = \left| \frac{x^{-s+1} \log x}{1-s} - \frac{x^{-s+1}}{(1-s)^2} \Big|_t^{t^2} \right| \\
&\leq 2 \log t \sqrt{\frac{t^2}{1/4 + t^2}} + \frac{t}{1/4 + t^2} + \frac{t^{1/2} \log t}{\sqrt{1/4 + t^2}} + \frac{t^{1/2}}{1/4 + t^2} \\
&\leq 2 \log t + \frac{1}{t} + \frac{\log t}{t^{1/2}} + \frac{1}{t^{3/2}} \leq 2 \log t + 0.921,
\end{align*}
\[
|I_2| + |I_3| = \left| \frac{((t)) \log t}{t^s} \right| + \left| \frac{((t^2)) \log t^2}{t^{2s}} \right| \leq \frac{1}{2} \left( \frac{\log t}{t^{1/2}} + \frac{2 \log t}{t} \right) \leq 0.639,
\]
\begin{equation}\label{40}
|I_4| = \left| \int_t^{t^2} \frac{((x))}{x^{s+1}} dx \right| = \left| \int_t^{t^2} \frac{((x)) \cos(t \log x)}{x^{\sigma+1}} dx - i \int_t^{t^2} \frac{((x)) \sin(t \log x)}{x^{\sigma+1}} dx \right|.
\end{equation}
By \cite{5,7},
\[
((x)) = -\frac{1}{\pi} \sum_{v=1}^{\infty} \frac{\sin 2\pi v x}{v}.
\]
According to \cite{6},
\[
\left| \int_a^b \frac{\cos(t \log x) \sin(2\pi v x)}{x^{1+\sigma}} dx \right| \leq \frac{8\pi v a}{a^\sigma (4\pi^2 v^2 a^2 - t^2)},
\]
\[
\left| \int_a^b \frac{\sin(t \log x) \sin(2\pi v x)}{x^{1+\sigma}} dx \right| \leq \frac{8\pi v a}{a^\sigma (4\pi^2 v^2 a^2 - t^2)}.
\]
Thus,
\begin{align*}
\left| \int_t^{t^2} \frac{((x)) \cos(t \log x)}{x^{\sigma+1}} dx \right| &\leq \sum_{v=1}^{\infty} \frac{1}{\pi v} \left| \int_t^{t^2} \frac{\sin(2\pi v x) \cos(t \log x)}{x^{\sigma+1}} dx \right| \\
&\leq \sum_{v=1}^{\infty} \frac{8}{t^{3/2} (4\pi^2 v^2 - 1)} \leq \frac{8}{t^{3/2}} \cdot \frac{1}{3.898\pi^2} \sum_{v=1}^{\infty} \frac{1}{v^2} \leq 0.018,
\end{align*}
and
\begin{align*}
\left| \int_t^{t^2} \frac{((x)) \sin(t \log x)}{x^{\sigma+1}} dx \right| &\leq \sum_{v=1}^{\infty} \frac{1}{\pi v} \left| \int_t^{t^2} \frac{\sin(2\pi v x) \sin(t \log x)}{x^{\sigma+1}} dx \right| \\
&\leq \sum_{v=1}^{\infty} \frac{8}{t^{3/2} (4\pi^2 v^2 - 1)} \leq 0.018.
\end{align*}
Therefore,
\[
|I_4| \leq 0.018 \sqrt{2}.
\]
Now,
\begin{equation}\label{41}
\begin{aligned}
|I_5| &= \left| s \int_t^{t^2} \frac{((x)) \log x}{x^{s+1}} dx \right| \\
&= \sqrt{1/4 + t^2} \left| \int_t^{t^2} \frac{((x)) \cos(t \log x) \log x}{x^{\sigma+1}} dx - i \int_t^{t^2} \frac{((x)) \sin(t \log x) \log x}{x^{\sigma+1}} dx \right|.
\end{aligned}
\end{equation}
By Lemma \ref{lemma2.2},
\[
\left| \int_a^b \frac{\cos(t \log x) \sin(2\pi v x) \log x}{x^{1+\sigma}} dx \right| \leq \frac{8\pi v a \log a}{a^\sigma (4\pi^2 v^2 a^2 - t^2)},
\]
\[
\left| \int_a^b \frac{\sin(t \log x) \sin(2\pi v x) \log x}{x^{1+\sigma}} dx \right| \leq \frac{8\pi v a \log a}{a^\sigma (4\pi^2 v^2 a^2 - t^2)}.
\]
So,
\begin{align*}
&\left| \sqrt{1/4 + t^2} \int_t^{t^2} \frac{((x)) \cos(t \log x) \log x}{x^{\sigma+1}} dx \right| \\
&\quad \leq \sqrt{1/4 + t^2} \sum_{v=1}^{\infty} \frac{1}{\pi v} \left| \int_t^{t^2} \frac{\sin(2\pi v x) \cos(t \log x) \log x}{x^{\sigma+1}} dx \right| \\
&\quad \leq 8 \sqrt{1/4 + t^2} \sum_{v=1}^{\infty} \frac{t^{1/2} \log t}{4\pi^2 v^2 t^2 - t^2} \leq 8 \sqrt{1 + \frac{1}{4t^2}} \sum_{v=1}^{\infty} \frac{\log t}{t^{1/2}} \frac{1}{4\pi^2 v^2 - 1} \leq 0.253,
\end{align*}
and similarly,
\[
\left| \sqrt{1/4 + t^2} \int_t^{t^2} \frac{((x)) \sin(t \log x) \log x}{x^{\sigma+1}} dx \right| \leq 0.253.
\]
Hence,
\[
|I_5| \leq 0.253 \sqrt{2}.
\]
Combining the estimates,
\begin{equation}\label{7}
\begin{aligned}
\left| \sum_{t < n \leq t^2} \frac{\log n}{n^s} \right| &\leq |I_1| + |I_2| + |I_3| + |I_4| + |I_5| \\
&\leq 2 \log t + 1.944.
\end{aligned}
\end{equation}
This completes the proof of Lemma \ref{lemma2.4}.
\end{proof}

\begin{lemma}\label{lemma2.5}
If $0 < a_1 \leq a_2 \leq a_3 \leq \cdots \leq a_n$, $x_i \in \mathbb{R}$, then
\[
\left| a_1 e^{i x_1} + a_2 e^{i x_2} + \cdots + a_n e^{i x_n} \right| \leq a_n \max\{1, b_2, b_3, \dots, b_n\},
\]
where
\begin{align*}
b_2 &= \max_{1 \leq i < j \leq n} \left| e^{i x_i} + e^{i x_j} \right|, \\
b_3 &= \max_{1 \leq i < j < k \leq n} \left| e^{i x_i} + e^{i x_j} + e^{i x_k} \right|, \\
&\vdots \\
b_n &= \left| e^{i x_1} + e^{i x_2} + \cdots + e^{i x_n} \right|.
\end{align*}
\end{lemma}

\begin{proof}
Consider $f(a_1, \dots, a_n) = \left| a_1 e^{i x_1} + a_2 e^{i x_2} + \cdots + a_n e^{i x_n} \right|$, with $0 \leq a_i \leq 1$. Note that
\begin{align*}
f(0, \dots, 0) &= 0, \\
f(0, \dots, 1) &= 1, \\
f(0, \dots, 1, \dots, 1, \dots, 0) &= |e^{i x_i} + e^{i x_j}|, \\
f(0, \dots, 1, \dots, 1, \dots, 1, \dots, 0) &= |e^{i x_i} + e^{i x_j} + e^{i x_k}|, \\
&\vdots
\end{align*}
Since
\[
g(a_1, \dots, a_n) = |f|^2 = \sum_{i=1}^n a_i^2 + 2 \sum_{i<j} a_i a_j \cos(x_i - x_j),
\]
and $g$ is a quadratic function in each $a_i$, its maximum on $[0,1]^n$ occurs at a vertex. Therefore,
\[
f(a_1, \dots, a_n) \leq \max\{1, b_2, b_3, \dots, b_n\}.
\]
This completes the proof of Lemma \ref{lemma2.5}.
\end{proof}

\section{Proof of Theorem \ref{theorem1.1}}
From \eqref{5}, for $s = 1/2 + it$,
\begin{equation}\label{74}
|\zeta'(1/2+it)| \leq \left| \sum_{n=1}^{[t]} \frac{\log n}{n^s} \right| + \left| \sum_{t < n \leq t^2} \frac{\log n}{n^s} \right| + |E(s)|.
\end{equation}
For the first term,
\begin{equation}\label{8}
\left| \sum_{n=1}^{[t]} \frac{\log n}{n^s} \right| \leq \int_0^t \frac{\log u}{u^{1/2}} du = 2 t^{1/2} (\log t - 2), \quad t \geq e^2.
\end{equation}
For $t \geq e^2$, the third term satisfies
\[
|E(s)| \leq \frac{2}{\sqrt{e^4 - 1}} + \sqrt{\frac{4e^4 + 1}{e^4 - 1}} (2 \log t + 2) + \frac{1}{e^2} + 2 \log t \leq 4.455 + 6.047 \log t.
\]
Combining with Lemma \ref{lemma2.4}, we obtain
\[
|\zeta'(1/2+it)| \leq 2 t^{1/2} \log t - 4 t^{1/2} + 8.047 \log t + 6.399.
\]
This completes the proof of Theorem \ref{theorem1.1}.

\section{Proof of Theorem \ref{theorem1.2}}
We further refine \eqref{5} as
\begin{equation}\label{10}
\begin{aligned}
\zeta'(s) &= -\sum_{n \leq t^{1/3}} \frac{\log n}{n^s} - \sum_{t^{1/3} < n \leq t^{2/3}} \frac{\log n}{n^s} - \sum_{t^{2/3} < n \leq t} \frac{\log n}{n^s} \\
&\quad - \sum_{t < n \leq t^2} \frac{\log n}{n^s} + E(s).
\end{aligned}
\end{equation}
We have
\begin{equation}\label{50}
\left| \sum_{n \leq t^{1/3}} \frac{\log n}{n^s} \right| \leq \int_0^{t^{1/3}} \frac{\log u}{u^{1/2}} du = 2 t^{1/6} \left( \frac{1}{3} \log t - 2 \right), \quad t \geq e^6.
\end{equation}
For $t \geq e^6$,
\begin{equation}\label{51}
|E(s)| \leq \frac{2}{\sqrt{e^{12} - 1}} + \sqrt{\frac{4e^{12} + 1}{e^{12} - 1}} (2 \log t + 2) + \frac{1}{e^6} + 2 \log t \leq 4.008 + 6.001 \log t.
\end{equation}
Next, we use exponential sum methods to estimate $\sum_{t^{1/3} < n \leq t^{2/3}} \frac{\log n}{n^s}$ and $\sum_{t^{2/3} < n \leq t} \frac{\log n}{n^s}$.

\subsection{Estimation of $\sum_{t^{2/3} < n \leq t} \frac{\log n}{n^s}$}
\begin{lemma}\label{lemma4.1} \cite[Lemma 4]{1}
Suppose $f$ is a real-valued function with two continuous derivatives on $[N+1, N+L]$. If there exist real numbers $V < W$, $W > 1$ such that
\begin{equation}\label{89}
\frac{1}{W} \leq |f''(x)| \leq \frac{1}{V} \quad \text{for } x \in [N+1, N+L],
\end{equation}
then
\[
\left| \sum_{n=N+1}^{N+L} e^{2\pi i f(n)} \right| \leq \frac{1}{5} \left( \frac{L}{V} + 1 \right) (8 W^{1/2} + 15).
\]
\end{lemma}

\begin{lemma}\label{lemma4.2}
Let $t \geq e^3$. Then
\[
\left| \sum_{t^{2/3} < n \leq t} \frac{\log n}{n^s} \right| \leq C_1 t^{1/6} \log t + C_2 t^{1/6} + C_3 \log t +\]
\[ C_4 + C_5 t^{-1/6} + C_6 t^{-1/2} + C_7 t^{-1/3} + C_8 t^{-2/3} + C_9 t^{-2/3} \log t + C_{10} t^{-1/2} \log t + C_{11} t^{-1/3} \log t.
\]
Where $C_i(k)$ ($i=1,\dots,11$) are constants, and $k$ is a parameter to be determined.
\end{lemma}

\begin{proof}
For $e^3 \leq t \leq t_1$ with $t_1$ sufficiently large,
\begin{equation}\label{100}
\left| \sum_{t^{2/3} < n \leq t} \frac{\log n}{n^s} \right| \leq \sum_{t^{2/3} < n \leq t} \frac{\log n}{n^{1/2}} \leq \int_1^t \frac{\log u}{u^{1/2}} du \leq 2 t_1^{1/2} (\log t_1 - 2) + 4.
\end{equation}
For $t \geq t_1$, let $k > 1$, $X_j = k^j t^{2/3}$, $N_j = [X_j]$ ($j=0,1,\dots,J$), with $X_j \leq t$, so that $J \leq \left[ \frac{\log t}{3 \log k} \right] + 1$. Then
\begin{equation}\label{12}
\left| \sum_{t^{2/3} < n \leq t} \frac{\log n}{n^s} \right| \leq \sum_{j=1}^J \left| \sum_{n=N_{j-1}+1}^{N_j} \frac{\log n}{n^s} \right|.
\end{equation}
Using Lemma \ref{lemma2.5} to estimate the inner sum,
\begin{equation}\label{13}
\left| \sum_{n=N_{j-1}+1}^{N_j} \frac{\log n}{n^s} \right| \leq \frac{\log(N_{j-1}+1)}{(N_{j-1}+1)^{1/2}} \max\{1, b_2, b_3, \dots, b_L\},
\end{equation}
where $L = N_j - N_{j-1}$.

Let $f(x) = -\frac{t \log x}{2\pi}$. For $X_{j-1} < N_{j-1}+1 \leq x \leq N_j \leq X_j$, we have
\[
\frac{t}{2\pi k^2 X_{j-1}^2} = \frac{t}{2\pi X_j^2} \leq |f''(x)| = \left| \frac{t}{2\pi x^2} \right| < \frac{t}{2\pi X_{j-1}^2}.
\]
Let $W = \frac{2\pi k^2 X_{j-1}^2}{t}$, $V = \frac{2\pi X_{j-1}^2}{t}$, and $L \leq (k-1) X_{j-1} + 1$. By Lemma \ref{lemma4.1},
\begin{equation}\label{14}
\begin{aligned}
\max\{1, b_2, \dots, b_L\} &\leq \frac{1}{5} \left( \frac{(k-1) t}{2\pi X_{j-1}} + \frac{t}{2\pi X_{j-1}^2} + 1 \right) \left( \frac{2^{7/2} \pi^{1/2} k X_{j-1}}{t^{1/2}} + 15 \right) \\
&= \frac{1}{5} \left( \frac{2^{5/2} k (k-1) t^{1/2}}{\pi^{1/2}} + \frac{2^{5/2} k t^{1/2}}{\pi^{1/2} X_{j-1}} + \frac{2^{7/2} \pi^{1/2} k X_{j-1}}{t^{1/2}} \right. \\
&\quad \left. + \frac{15 (k-1) t}{2\pi X_{j-1}} + \frac{15 t}{2\pi X_{j-1}^2} + 15 \right).
\end{aligned}
\end{equation}
Substituting \eqref{14} into \eqref{13} and then into \eqref{12}, we get
\[
\left| \sum_{t^{2/3} < n \leq t} \frac{\log n}{n^s} \right| \leq \frac{1}{5} \sum_{j=1}^J \frac{\log(N_{j-1}+1)}{(N_{j-1}+1)^{1/2}} \left( \frac{2^{5/2} k (k-1) t^{1/2}}{\pi^{1/2}} +\right.\]
\[\left. \frac{2^{5/2} k t^{1/2}}{\pi^{1/2} X_{j-1}} + \frac{2^{7/2} \pi^{1/2} k X_{j-1}}{t^{1/2}} + \frac{15 (k-1) t}{2\pi X_{j-1}} + \frac{15 t}{2\pi X_{j-1}^2} + 15 \right).
\]
Since $X_{j-1} < N_{j-1}+1 < X_j$, this becomes
\begin{equation}\label{15}
\begin{aligned}
\left| \sum_{t^{2/3} < n \leq t} \frac{\log n}{n^s} \right| &\leq \frac{1}{5} \sum_{j=1}^J \left( \frac{2^{5/2} k (k-1) t^{1/2}}{\pi^{1/2}} \frac{\log X_{j-1}}{X_{j-1}^{1/2}} + \frac{2^{5/2} k t^{1/2}}{\pi^{1/2}} \frac{\log X_{j-1}}{X_{j-1}^{3/2}} \right. \\
&\quad + \frac{2^{7/2} \pi^{1/2} k}{t^{1/2}} X_{j-1}^{1/2} \log X_{j-1} + \frac{15 (k-1) t}{2\pi} \frac{\log X_{j-1}}{X_{j-1}^{3/2}} \\
&\quad \left. + \frac{15 t}{2\pi} \frac{\log X_{j-1}}{X_{j-1}^{5/2}} + 15 \frac{\log X_{j-1}}{X_{j-1}^{1/2}} \right).
\end{aligned}
\end{equation}
Let
\begin{align*}
M_1 &:= \sum_{j=1}^J X_{j-1}^{1/2} \log X_{j-1}, \\
M_2(\delta) &:= \sum_{j=1}^J \frac{\log X_{j-1}}{X_{j-1}^{\delta/2}} \quad (\delta = 1, 3, 5).
\end{align*}
Since $J < \frac{\log t}{3 \log k} + 1$, we can bound $M_1$ and $M_2(\delta)$ as in \eqref{16} and \eqref{17}. Substituting these bounds into \eqref{15} yields the desired result.
\end{proof}

\subsection{Estimation of $\sum_{t^{1/3} < n \leq t^{2/3}} \frac{\log n}{n^s}$}
\begin{lemma}\label{lemma4.3} \cite[Lemma 5]{1}
Let $f(n)$ be a real-valued function and $M$ a positive integer. Then
\[
\left| \sum_{n=N+1}^{N+L} e^{2\pi i f(n)} \right|^2 \leq \frac{(L+M)L}{M} + \frac{2(L+M)}{M} \sum_{m=1}^{M-1} \left(1 - \frac{m}{M}\right) \max_{K \leq L} |S'_m(K)|,
\]
where
\[
S'_m(K) = \sum_{n=N+1}^{N+L} e^{2\pi i [f(n+m) - f(n)]}.
\]
\end{lemma}

\begin{lemma}\label{lemma4.4} \cite[p. 1275]{1}
For each $j$,
\begin{equation}\label{21}
\begin{aligned}
|S_j| &\leq \frac{((\tau - 1) X_{j-1} + 1 + q t_2^{-1/3} X_{j-1})^{1/2} ((\tau - 1) X_{j-1} + 1)^{1/2}}{M^{1/2}} \\
&\quad + \frac{2^{1/2} (\tau X_{j-1} + 1)^{1/2}}{M^{1/2}} \left( \sum_{m=1}^{M-1} \left(1 - \frac{m}{M}\right) \max_{K \leq L} |S'_m(K)| \right)^{1/2},
\end{aligned}
\end{equation}
where
\[
S'_m(K) = \sum_{n=N_{j-1}+1}^{N_{j-1}+K} e^{-i t [\log(n+m) - \log n]}.
\]
\end{lemma}

\begin{lemma}\label{lemma4.5} \cite[p. 1276]{1}
Let $V = \frac{\pi X_{j-1}^3}{m t}$, $W = \frac{\pi (\tau + 1)^3 X_{j-1}^3}{m t}$. Then
\begin{equation}\label{22}
\begin{aligned}
\max_{K \leq L} |S'_m(K)| &\leq \frac{1}{5} \left( \frac{8 (\tau - 1) (\tau + 1)^{3/2} m^{1/2} t^{1/2}}{\pi^{1/2} X_{j-1}^{1/2}} + \frac{8 (\tau + 1)^{3/2} m^{1/2} t^{1/2}}{\pi^{1/2} X_{j-1}^{3/2}} \right. \\
&\quad + \frac{8 \pi^{1/2} (\tau + 1)^{3/2} X_{j-1}^{3/2}}{m^{1/2} t^{1/2}} + \frac{15 (\tau - 1) m t}{\pi X_{j-1}^2} \\
&\quad \left. + \frac{15 m t}{\pi X_{j-1}^3} + 15 \right).
\end{aligned}
\end{equation}
\end{lemma}

\begin{lemma}\label{lemma4.6} \cite[p. 1277]{1}
For $0 \leq m \leq M$,
\begin{align*}
\sum_{m=1}^{M-1} \left(1 - \frac{m}{M}\right) m^{1/2} &\leq \frac{4}{15} M^{3/2}, \\
\sum_{m=1}^{M-1} \left(1 - \frac{m}{M}\right) m^{-1/2} &\leq \frac{4}{3} M^{1/2}, \\
\sum_{m=1}^{M-1} \left(1 - \frac{m}{M}\right) m &= \frac{1}{6} M^2, \\
\sum_{m=1}^{M-1} \left(1 - \frac{m}{M}\right) &= \frac{1}{2} M.
\end{align*}
\end{lemma}

\begin{lemma}\label{lemma4.7}
Let $t \geq e^6$. Then
\[
\left| \sum_{t^{1/3} < n \leq t^{2/3}} \frac{\log n}{n^s} \right| \leq c_1 t^{1/6} (\log t)^2 + c_2 t^{1/6} \log t + c_3 t^{1/6} + c_4 (\log t)^2 + c_5 \log t + c_6,
\]
where $c_1(\tau,q,t_2), c_2(\tau,q,t_2), c_3(\tau,q,t_2), c_4(\tau,t_2), c_5(\tau,q,t_2), c_6(\tau,q,t_2)$ are constants, and $\tau, q, t_2$ are parameters to be determined.
\end{lemma}

\begin{proof}
For $e^6 \leq t \leq t_2$ with $t_2$ sufficiently large,
\[
\left| \sum_{t^{1/3} < n \leq t^{2/3}} \frac{\log n}{n^s} \right| \leq \sum_{1 < n \leq t^{2/3}} \frac{\log n}{n^{1/2}} \leq \int_1^{t^{2/3}} \frac{\log u}{u^{1/2}} du \leq 2 t_2^{1/3} \left( \frac{2}{3} \log t_2 - 2 \right) + 4.
\]
For $t \geq t_2$, let $\tau > 1$, $X_j = \tau^j t^{1/3}$ ($j=0,1,\dots,J$), $N_j = [X_j]$, with $X_j \leq t^{2/3}$, so that $J \leq \left[ \frac{\log t}{3 \log \tau} \right] + 1 < \frac{\log t}{3 \log \tau} + 1$. Then
\begin{equation}\label{19}
\left| \sum_{t^{1/3} < n \leq t^{2/3}} \frac{\log n}{n^s} \right| \leq \sum_{j=1}^J \left| \sum_{n=N_{j-1}+1}^{N_j} \frac{\log n}{n^s} \right|.
\end{equation}
Using Lemma \ref{lemma2.5},
\begin{equation}\label{20}
\left| \sum_{n=N_{j-1}+1}^{N_j} \frac{\log n}{n^s} \right| \leq \frac{\log(N_{j-1}+1)}{(N_{j-1}+1)^{1/2}} \max\{1, b_2, b_3, \dots, b_L\},
\end{equation}
where $L = N_j - N_{j-1} \leq (\tau - 1) X_{j-1} + 1$.

Let $S_j = \max\{1, b_2, \dots, b_L\}$, $f(x) = -\frac{t \log x}{2\pi}$, and $M = \frac{q X_{j-1}}{t^{1/3}}$ with $q \geq 2$. Using Lemmas \ref{lemma4.3}--\ref{lemma4.6}, we obtain an upper bound for $S_j$. Substituting this into \eqref{20} and then into \eqref{19}, and using bounds for sums over $j$, we obtain the desired result.
\end{proof}

\subsection{Estimation of $|\zeta'(1/2+it)|$}
Combining the results above, for $t \geq e^6$ we have
\begin{equation}\label{70}
|\zeta'(1/2+it)| \leq Q_1 t^{1/6} (\log t)^2 + Q_2 t^{1/6} \log t + Q_3 t^{1/6} + Q_4 (\log t)^2 + Q_5 \log t + Q_6,
\end{equation}
where the $Q_i$ are as given in the theorem.

This completes the proof of Theorem \ref{theorem1.2}.

\section{Conclusion}
Since Titchmarsh's 1985 result
\[
\zeta'(\sigma + it) \le \exp\left( \frac{C \log t}{\log \log t} \right), \quad \sigma \geq 1/2, \quad t \geq 10,
\]
the study of upper bounds for the derivative of the zeta function has seen limited progress. However, recently there have been many estimates for the absolute value of the zeta function on the 1-line and 1/2-line, see references \cite{2,9,12,15,19}. In this paper, using exponential sum methods, we systematically derive a formula for effective estimates of $|\zeta'(1/2+it)|$, further refining Titchmarsh's result. Our results may not be optimal, and seeking the best possible bounds and applying the method to higher derivatives of the zeta function are worthwhile directions for future research.

\end{document}